\documentclass{amsart}
\usepackage{mathrsfs}
\usepackage{amsmath, amsthm, amscd, amsfonts, amssymb, graphicx, color}
\usepackage[bookmarksnumbered, colorlinks, plainpages]{hyperref}

\newtheorem{theorem}{Theorem}[section]
\newtheorem{lemma}[theorem]{Lemma}

\newtheorem{corollary}[theorem]{Corollary}
\newtheorem{definition}[theorem]{Definition}

\theoremstyle{remark}

\numberwithin{equation}{section}

\begin{document}

\title[Noncommutative functional calculate and its application]{Noncommutative functional calculate and its application}

\author{Lvlin Luo}
\address{School of Mathematics and Statistics, Xidian University, 710071, Xi'an, P. R. China}
\address{School of Mathematical Sciences, Fudan University, 200433, Shanghai, P. R. China}
\email{luoll12@mails.jlu.edu.cn}
\email{luolvlin@fudan.edu.cn}



\date{}



\begin{abstract}
In this paper we construct an unitary operator $F_{xx*}$ such that $(F_{xx^{*}})^2=identity$ and $Fix(F_{xx^*})\neq\emptyset$.
We get the unitary equivalent representations $F_{xx*}(M_{z\psi(z)}-a)$ on $\mathcal{L}^{2}(\sigma(|T+a|),\mu_{|T+a|})$ for any given $T\in\mathcal{B}(\mathbb{H})$,
where $\psi(z)\in\mathcal{L}^{\infty}(\sigma(|T+a|),\mu_{|T+a|})$,
$a\in\rho(T)$, $F_{xx*}(f(xx^*))=f(x^*x)$,
$\mathcal{B}(\mathbb{H})$ is the set of all bounded linear operator on complex separable Hilbert space $\mathbb{H}$.
Also, we get that if $z\psi(z)\in Fix(F_{xx^*})$, then $T$ has a nontrivial invariant subspace space on
$\mathbb{H}$ which has dimension $>1$.
Moreover, we define the Lebesgue class $\mathcal{B}_{Leb}(\mathbb{H})\subset\mathcal{B}(\mathbb{H})$
and get that if $T$ is a Lebesgue operator,
then $T$ is Li-Yorke chaotic if and only if $T^{*-1}$ is.
\end{abstract}

\maketitle

\section{Introduction}

Early, the invariant subspace problem, that is, whether an linear operator on a certain linear space has a nontrivial invariant subspace, have been stated by Beurling and vonNeumann \cite{JohannesKaiser2016}.
For finite dimensional vector spaces or nonseparable Hilbert spaces, the result is trivial.
But for infinite separable Hilbert spaces the problem is unsolved in a long time.
From then on, there are a lot of results about this open problem and many mathematicians answered that under conditions on the operator or on the space.

In 1966, Allen R. Bernstein and Abraham Robinson show that if $T$ is a bounded linear operator on a complex Hilbert space $\mathbb{H}$ and $p$ is a non-zero polynomial such that $p(T)$ is compact,
then $T$ has nontrivial invariant subspace \cite{AllenRBernsteinAbrahamRobinson1966}.
Especially, when $p(t)=t$, that is, $T$ itself is compact, this was proved independently by von Neumann and N. Aronszajn and in 1954, this result extended to apply to a compact on Banach space by N. Aronszajn and K. T. Smith \cite{NAronszajnKTSmith1954}.

In 1973, V. I. Lomonosov showed that on Banach space if $T$ is not a scalar multiple of the identity and commutes with a nonzero compact operator, then $T$ has a nontrivial hyperinvariant subspace\cite{VILomonosov1973},
that is, any bounded linear operator commuting with $T$ has a nontrivial invariant subspace.

In 1976, Per H. Enflo firstly constructed an operator on a Banach space having no nontrivial invariant subspace
\cite{PerH.Enflo1976}.
In 1983, Aharon Atzmon constructed a nuclear Fr\'{e}chet space $\mathbb{F}$ and a bounded linear operator which has no nontrivial invariant subspace \cite{AharonAtzmon1983}.
And in 1984, C. J. Read \cite{CJRead1984} made a example such that there is a bounded linear operator without nontrivial invariant subspace on $\ell_1$ .

In 2011, A. Argyros and G. Haydon constructed the first example of Banach space for which every bounded linear operator on the space has the form $\lambda+K$ where $\lambda$ is a real scalar and $K$ is a compact operator, such that every bounded linear operator on the space has a nontrivial invariant subspace \cite{SpirosAArgyrosRichardGHaydon2011}.

For research the invariant subspace problem and with the always developing on dynamics,
Operator Dynamics or Linear Dynamics is a rapidly evolving branch of functional analysis,
which was probably born in 1982 with the Toronto Ph. D. thesis of C. Kitai \cite{CKitai1982}.
It has become rather popular because of the efforts of many mathematicians survey,
such as \cite{GGodefroy2003,JHShapiro2001,KGGrosseErdmann1999,FBayartEMatheron2009,KGGrosseErdmannAPerisManguillot}.

Interdiscipline becomes a common phenomenon in all areas. And there is new objects about $C^{*}$-algebras and dynamics,
In\cite{RobertSDoran1994}, Gelfand-Naimark theorem\cite{IzrailMGelfandMarkANaimark1943} is called ``the fundamental theorem of $C^{*}$-algebras". Such that Ichiro Fujimoto said that this theorem eventually opened the gate to the subject of $C^{*}$-algebras\cite{IchiroFujimoto1998}. Hence there are various attempts to generalize
this theorem\cite{JMGFell1961,MTakesaki1967,EMAlfsen1968,Ichiro Fujimoto1971,PKruszynskiSLWoronowicz1982}.

For invertible dynamics, the relationship between $(X,f)$ and $(X,f^{-1})$ is raised by Stockman as an open question\cite{DStockman2012}. And \cite{LuoLvlinHouBingzhe2015,HouBingzheLuoLvlin2016} and \cite{LuoLvlinHouBingzhe2016} give counter examples with some chaos for this question in noncompact spaces and compact spaces, respectively. For an invertible bounded linear operator $T$ on Hilbert spaces $\mathbb{H}$,
the dynamics relationship between $(\mathbb{H},T)$ and $(\mathbb{H},T^{*-1})$ is also interested. Any way, counter example is trivial for some chaos between $T$ and $T^{*-1}$.
But there is no generally way to this research.
In fact, the $C^{*}$-algebra $\mathcal{A}(T)$ generated by $T$ is not obvious effect for that.
So we get another way,
that is representation theory for $T$ by $C^{*}$-algebra, and which is one step to get some dynamics relationship between $T$ and $T^{*-1}$.

In this paper, $\mathbb{H}$ denote separable Hilbert space over $\mathbb{C}$,
$\mathcal{B}(\mathbb{H})$ denote the set of all bounded linear operator on $\mathbb{H}$.
For any given $T\in\mathcal{B}(\mathbb{H})$,
we construct $C^{*}$-algebra induced by polar decomposition $T=U|T|$ and we construct an unitary operator $F_{xx*}$ such that $(F_{xx^{*}})^2=identity$ and $Fix(F_{xx^*})\neq\emptyset$.

We get that for any given $T\in\mathcal{B}(\mathbb{H})$,
$F_{xx*}(M_{z\psi(z)}-a)$ is the unitary equivalent representations of $T$ on $\mathcal{L}^{2}(\sigma(|T+a|),\mu_{|T+a|})$,
where $\psi(z)\in\mathcal{L}^{\infty}(\sigma(|T+a|),\mu_{|T+a|})$,
$a\in\rho(T)$, $F_{xx*}(f(xx^*))=f(x^*x)$.
Also, we get that if $z\psi(z)\in Fix(F_{xx^*})$, then $T$ has a nontrivial invariant subspace space on
$\mathbb{H}$ which has dimension $>1$.
Moreover, we define the Lebesgue class $\mathcal{B}_{Leb}(\mathbb{H})\subset\mathcal{B}(\mathbb{H})$
and get that if $T$ is a Lebesgue operator,
then $T$ is Li-Yorke chaotic if and only if $T^{*-1}$ is.
In fact, we get that $\mathcal{B}_{Leb}(\mathbb{H})\cap\mathcal{B}_{Nor}(\mathbb{H})\neq\emptyset$ and
$\mathcal{B}_{Leb}(\mathbb{H})\cap(\mathcal{B}(\mathbb{H})\setminus\mathcal{B}_{Nor}(\mathbb{H}))\neq\emptyset$,
where $\mathcal{B}_{Nor}(\mathbb{H})$ is the set of all normal operators on $\mathbb{H}$.

\section{Decomposition and isomorphic representation for $\mathbb{H}$}

Let $X$ be a compact subset of $\mathbb{C}$.
Let $\mathcal{C}(X)$ denote the set of all continuous function on the compact space $X$
and let $\mathcal{P}(x)$ denote the set of all polynomial on $X$.
For any given $T\in\mathcal{B}(\mathbb{H})$,
let $\sigma(T)$ be its spectrum.
By the Polar Decomposition theorem \cite{JohnBConway2000}P15 we get $T=U|T|$ and $|T|^2=T^{*}T$.
If $T$ is invertible, then $U$ is an unitary operator.
Let $\mathcal{A}(|T|)$ denote the $C^{*}$-algebra generated by $|T|$ and $1$.

\begin{definition}\label{liyorkehundundedingyi1}
Let $T\in\mathcal{B}(\mathbb{H})$,if there exists $x\in\mathbb{H}$ satisfies:

\begin{eqnarray*}
&&(1) \varlimsup\limits_{n\to\infty}|T^{n}(x)\|>0;\\
&&(2) \varliminf\limits_{n\to\infty}\|T^{n}(x)\|=0.
\end{eqnarray*}

Then we say that $T$ is Li-Yorke chaotic,and named $x$ is a Li-Yorke chaotic point of $T$,where $x\in\mathbb{H},n\in\mathbb{N}$.
\end{definition}

\begin{lemma}\label{weierstrassnikefenyinli1}
Let $X\subseteq\mathbb{C}$ be a compact subset not containing zero.
If $\mathcal{P}(x)$ is dense in $\mathcal{C}(X)$,
then $\mathcal{P}(\frac{1}{x})$ is also dense in $\mathcal{C}(X)$.
\end{lemma}
\begin{proof}
By the property of polynomial we know that $\mathcal{P}(\frac{1}{x})$ is a algebraic closed subalgebra of $\mathcal{C}(X)$ and we get:

(1) $1\in\mathcal{P}(\frac{1}{x})$;

(2) $\mathcal{P}(\frac{1}{x})$ separate the points of $X$;

(3) If $p(\frac{1}{x})\in\mathcal{P}(x)$,then $\bar{p}(\frac{1}{x})\in\mathcal{P}(x)$.

By the Stone-Weierstrass theorem \cite{JohnBConway1990}P145 we get the conclusion.
\end{proof}

\begin{lemma}\label{weierstrasspingfangkefenyinli2}
Let $X\subseteq\mathbb{R_+}$.
If $\mathcal{P}(|x|)$ is dense in $\mathcal{C}(X)$,
then $\mathcal{P}(|x|^2)$ is also dense in $\mathcal{C}(X)$.
\end{lemma}
\begin{proof}
For $X\subseteq\mathbb{R_+}$,
$x\neq y$
$\Longleftrightarrow x^2\neq y^2$.
By Lemma $\ref{weierstrassnikefenyinli1}$ we get the conclusion.
\end{proof}

By the GNS construction \cite{JohnBConway1990}P250, for the $C^{*}$-algebra $\mathcal{A}(|T|)$,
we have the following decomposition.

\begin{lemma}\label{gnsfenjiedingliyingyongyinli3}
Let $T$ be an invertible bounded linear operator on $\mathbb{H}$,
$\mathcal{A}(|T|)$ be the complex $C^{*}$-algebra generated by $|T|$ and $1$.
There is a sequence of nonzero $\mathcal{A}(|T|)$-invariant subspace.
$\mathbb{H}_1,\mathbb{H}_2,\cdots$ such that:

(1) $\mathbb{H}=\mathbb{H}_1\bigoplus\mathbb{H}_2\bigoplus\cdots$;

(2) For every $\mathbb{H}_i$,
there is a $\mathcal{A}(|T|)$-cyclic vector $\xi^i$ such that $\mathbb{H}_i=\overline{\mathcal{A}(|T|)\xi^i}$ and
$|T|\mathbb{H}_i=\mathbb{H}_i=|T|^{-1}\mathbb{H}_i$.
\end{lemma}
\begin{proof}
The decomposition of (1) is obvious by \cite{WilliamArveson2002}P54.
Hence $|T|\mathbb{H}_i\subseteq\mathbb{H}_i$,
that is $\mathbb{H}_i\subseteq|T|^{-1}\mathbb{H}_i$.
By Lemma $\ref{weierstrassnikefenyinli1}$ we get $|T|^{-1}\mathbb{H}_i\subseteq\mathbb{H}_i$.
Therefore $|T|\mathbb{H}_i=\mathbb{H}_i=|T|^{-1}\mathbb{H}_i$.
\end{proof}

Let $\xi\in\mathbb{H}$ is a $\mathcal{A}(|T|)$-cyclic vector such that $\mathcal{A}(|T|)\xi$ is dense in $\mathbb{H}$.
Because of the spectrum $\sigma{(|T|)}\neq\emptyset$,
on $\mathcal{C}(\sigma(|T|))$ define the non-zero linear functional

$$
\rho_{|T|}: \rho_{|T|}(f)=<f(|T|)\xi,\xi>,\forall f\in\mathcal{C}(\sigma(|T|)).
$$

Then $\rho_{|T|}$ is a positive linear functional,
by \cite{WilliamArveson2002}P54 and the Riesz-Markov theorem,
on $\mathcal{C}(\sigma(|T|))$ we get that there exists an unique finite positive Borel measure $\mu_{|T|}$ such that

$$
\left.\begin{array}{lr}
\int\limits_{\sigma(|T|)}f(z)\,d\mu_{|T|}(z)=<f(|T|)\xi,\xi>, & \forall f\in\mathcal{C}(\sigma(|T|)).
\end{array}\right.
$$

\begin{theorem}\label{hanshuyansuan5}
Let $T$ be an invertible bounded linear operator on $\mathbb{H}$,
$\mathcal{A}(|T^n|)$ be the complex $C^{*}$-algebra generated by $|T^n|$ and $1$ and
let $\xi_n$ be a $\mathcal{A}(|T^n|)$-cyclic vector such that $\overline{\mathcal{A}(|T^n|)\xi_n}=\mathbb{H}$,
where $n\in\mathbb{N}$.
Then:

(1) there is an unique positive linear functional

$$
\left.\begin{array}{lr}
\int\limits_{\sigma(|T^n|)}{f(z)\,d\mu_{|T^n|}(z)}=<f(|T^n|)\xi_n,\xi_n>, & \forall f\in\mathcal{L}^{1}(\sigma(|T^n|),\mu_{|T^n|}).
\end{array}\right.
$$

(2) there is an unique finite positive complete Borel measure $\mu_{|T^n|}$ such that
$\mathcal{L}^{2}(\sigma(|T^n|),\mu_{|T^n|})$ is an isomorphic representation of $\mathbb{H}$.
\end{theorem}
\begin{proof}

(1) For $\mathcal{A}(|T^n|)$-cyclic vector $\xi_n$, define the linear functional

$$
\rho_{|T^n|}(f)=<f(|T^n|)\xi_n,\xi_n>,\forall f\in\mathcal{C}(\sigma(|T|)).
$$

We get that on $\mathcal{C}(\sigma(|T^n|))$ there is an unique finite positive Borel measure $\mu_{|T^n|}$ such that

$$
\left.\begin{array}{lr}
\int\limits_{\sigma(|T^n|)}{f(z)\,d\mu_{|T^n|}(z)}=<f(|T^n|)\xi_n,\xi_n>, & \forall f\in\mathcal{C}(\sigma(|T^n|)).
\end{array}\right.
$$

Moreover we can complete the Borel measure $\mu_{|T^n|}$ on $\sigma(|T^n|)$,
also using $\mu_{|T^n|}$ to denote the complete Borel measure,
by \cite{PaulRHalmos1974} we know that the complete Borel measure is uniquely.

For any $f\in\mathcal{L}^{2}(\sigma(|T^n|),\mu_{|T^n|})$, because of

$$
\left.\begin{array}{l}
\rho_{|T^n|}(|f|^2)=\rho_{|T^n|}(\bar{f}f)=<f(|T^n|)^{*}f(|T^n|)\xi_n,\xi_n>=\|f(|T^n|)\xi_n\|\geq0.
\end{array}\right.
$$

we get that $\rho_{|T^n|}$ is a positive linear functional, that is (1).

(2) We know that $\mathcal{C}(\sigma(|T^n|))$ is dense in $\mathcal{L}^{2}(\sigma(|T^n|),\mu_{|T^n|})$.
For any $f,g\in\mathcal{C}(\sigma(|T^n|))$ we get

\begin{eqnarray*}
&&<f(|T^n|)\xi_n,g(|T^n|)\xi_n>_{\mathbb{H}}\\
&&=<g(|T^n|)^{*}f(|T^n|)\xi_n,\xi_n>\\
&&=\rho_{|T^n|}(\bar{g}f)=\int\limits_{\sigma(|T^n|)}{f(z)\bar{g}(z)\,d\mu_{|T^n|}(z)}\\
&&=<f,g>_{\mathcal{L}^{2}(\sigma(|T^n|),\mu_{|T^n|})}.
\end{eqnarray*}

Therefore $R_0:\mathcal{C}(\sigma(|T^n|))\to\mathbb{H},f(z)\to f(|T^n|)\xi_n$ is a surjection isometry from
$\mathcal{C}(\sigma(|T^n|))$ to $\mathcal{A}(|T^n|)\xi_n$,
also $\mathcal{C}(\sigma(|T^n|))$ and $\mathcal{A}(|T^n|)\xi_n$ is dense subspaces of
$\mathcal{L}^{2}(\sigma(|T^n|),\mu_{|T^n|})$ and $\mathbb{H}$, respectively.
Its closed extension
$$R_{|T^n|}:\mathcal{L}^{2}(\sigma(|T^n|),\mu_{|T^n|})\to\mathbb{H},f(z)\to f(|T^n|)\xi_n$$
is an unitary operator.
Hence $R_{|T^n|}$ is the unique unitary operator induced by the unique finite positive complete Borel measure $\mu_{|T^n|}$ such that
$R_{|T^n|}$
is an isomorphic representation of the Hilbert space $\mathbb{H}$.
\end{proof}

\section{Unitary equivalent representation for $T$}

Following the above part,
we get an unitary operator from
$\mathcal{L}^{2}(\sigma(|T^{-1}|),\mu_{|T^{-1}|})$
to
$\mathcal{L}^{2}(\sigma(|T|),\mu_{|T|})$ by the composition $R^{-1}_{|T|}\circ R_{|T^{-1}|}$.
However,
there is nothing more interesting information from $R^{-1}_{|T|}\circ R_{|T^{-1}|}$.

We know that spectrum theory and functional calculate of normal operator is very important in study operator theory and $C^*$-algebra.
Inspired by Hua Loo-kang theorem on the automorphisms of a sfield,
in this part we give a useful structure describing from
$\mathcal{L}^{2}(\sigma(|T^{-1}|),\mu_{|T^{-1}|})$
to
$\mathcal{L}^{2}(\sigma(|T|^{-1}),\mu_{|T|^{-1}})$.
Combining this structure describing and functional calculate of normal operator,
we give a unitary equivalent representation for any given bounded linear operator $T$ on $\mathbb{H}$,
i.e. give functional calculate for any given bounded linear operator $T$ on $\mathbb{H}$.

\begin{lemma}\label{puxiangdengdedingyi5}
Let $T$ be an invertible bounded linear operator on $\mathbb{H}$,
then
$$
\sigma(|T^{-1}|)=\sigma(|T|^{-1}).
$$
\end{lemma}
\begin{proof}
Because of
$$
\lambda\in\sigma(T^{*}T)\Longleftrightarrow\frac{1}{\lambda}\in\sigma(T^{*-1}T^{-1}),
$$
we get
$$
\lambda\in\sigma(|T|)\Longleftrightarrow\frac{1}{\lambda}\in\sigma(|T^{-1}|),
$$
that is,
$$
\sigma(|T^{-1}|)=\sigma(|T|^{-1}).
$$
\end{proof}

\begin{definition}\label{yousuanzidexingshidingyi5}
Let $T$ be an invertible bounded linear operator on $\mathbb{H}$,
on $\sigma(|T|^{-1})$ define the function
$$
F_{xx^{*}}:\mathcal{P}(xx^{*})\rightarrow\mathcal{P}(x^{*}x),F_{xx^{*}}(f(xx^{*}))=f(x^{*}x).
$$
Then using Theorem~\ref{hanshuyansuan5},
for the isomorphic representations $R_{|T|^{-1}}$ and $R_{|T^{-1}|}$ of the Hilbert space $\mathbb{H}$,
define
$$
F_{xx^{*}}:
   \mathcal{L}^{2}(\sigma(|T|^{-1}),\mu_{|T|^{-1}})\rightarrow\mathcal{L}^{2}(\sigma(|T^{-1}|),\mu_{|T^{-1}|}).
$$
Obviously, $F_{xx^{*}}$ induced an linear operator $F_{xx^{*}}^{\mathbb{H}}$ on $\mathbb{H}$.
\end{definition}

By the Polar Decomposition theorem $T=U|T|$,
we get $U^{*}T^{*}TU=TT^{*}$ and $U^{*}|T|^{-2}U=|T^{-2}|$.
In fact, when $T$ is invertible,
we can choose an specially unitary operator such that $|T|^{-1}$ and $|T^{-1}|$ are unitary equivalent.
We have the following unitary equivalent by Theorem $\ref{hanshuyansuan5}$.

\begin{theorem}\label{TjueduizhiniyuTnijueduizhideguanxi15}
Let $T$ be an invertible bounded linear operator on $\mathbb{H}$ and
let $\mathcal{A}(|T|)$ be the complex $C^{*}$-algebra generated by $|T|$ and $1$.
We get that $F_{xx^{*}}^{\mathbb{H}}$ is an unitary operator and
$$F_{xx^{*}}^{\mathbb{H}}|T|^{-1}=|T^{-1}|F_{xx^{*}}^{\mathbb{H}}.$$

Moreover the unitary operator $F_{xx^{*}}^{\mathbb{H}}$ is induced by an almost everywhere non-zero function $|\phi_{|T|}(z)|$, where $|\phi_{|T|}(z)|\in\mathcal{L}^{1}(\sigma(|T|),\mu_{|T|})$,
that is,
$$d\,\mu_{|T^{-1}|}=|\phi_{|T|}(\frac{1}{z})|d\,\mu_{|T|^{-1}}.$$
\end{theorem}
\begin{proof}
By Lemma $\ref{gnsfenjiedingliyingyongyinli3}$, lose no generally, let $\xi_{|T|}$ be a $\mathcal{A}(|T|)$-cyclic vector such that $\mathbb{H}=\overline{\mathcal{A}(|T|)\xi_{|T|}}$.

(1) On $\sigma(|T|)$, define the function

$$
F_{z^{-1}}:\mathcal{P}(z)\rightarrow\mathcal{P}(z^{-1}),F_{z^{-1}}(f(z))=f(z^{-1}),
$$

it is easy to find that $F_{z^{-1}}$ is linear. Because of

$$
\int\limits_{\sigma(|T|)}{f(z^{-1})d\mu_{|T|}(z)}=<f(|T|^{-1})\xi,\xi>=\int\limits_{\sigma(|T|^{-1})}{f(z)d\mu_{|T|^{-1}}(z)}.
$$

We get $d\mu_{|T|^{-1}}(z)=|z|^2d\mu_{|T|}(z)$. Hence

\begin{eqnarray*}
&&\|F_{z^{-1}}(f(z))\|^2_{\mathcal{L}^2(\sigma(|T|^{-1}),\mu_{|T|^{-1}})}\\
&&=\int\limits_{\sigma(|T|^{-1})}{F_{z^{-1}}(f(z))\bar{F}_{z^{-1}}(f(z))d\mu_{|T|^{-1}}(z)}\\
&&=\int\limits_{\sigma(|T|^{-1})}{f(z^{-1})\bar{f}(z^{-1})d\mu_{|T|^{-1}}(z)}\\
&&=\int\limits_{\sigma(|T|)}{|z|^2f(z)\bar{f}(z)d\mu_{|T|}(z)}\\
&&\leq \sup\limits_{m\in\sigma(|T|)}m^2\int\limits_{\sigma(|T|)}{f(z)\bar{f}(z)d\mu_{|T|}(z)}\\
&&\leq\sup\limits_{m\in\sigma(|T|)}m^2\|f(z)\|^2_{\mathcal{L}^2(\sigma(|T|),\mu_{|T|})}.
\end{eqnarray*}

Therefore $\|F_{z^{-1}}\|\leq\sup\limits_{m\in\sigma(|T|)}|m|$.
And following the Banach Inverse Mapping theorem \cite{JohnBConway1990}P91,
we get that $F_{z^{-1}}$ is an invertible bounded linear operator from
$\mathcal{L}^2(\sigma(|T|),\mu_{|T|})$ to $\mathcal{L}^2(\sigma(|T|^{-1}),\mu_{|T|^{-1}})$.

Define the operator

$$
F_{z^{-1}}^{\mathbb{H}}:\mathcal{A}(|T|)\xi\rightarrow\mathcal{A}(|T|^{-1})\xi,F(f(|T|)\xi)=f(|T|^{-1})\xi.
$$

By Lemma $\ref{weierstrassnikefenyinli1}$ and \cite{WilliamArveson2002}P55,
we get that $F_{z^{-1}}^{\mathbb{H}}$ is an invertible bounded linear operator on the Hilbert space $\overline{\mathcal{A}(|T|)\xi}=\mathbb{H}$
and $\|F_{z^{-1}}^{\mathbb{H}}\|\leq\sup\sup\limits_{m\in\sigma(|T|)}|m|$. Moreover we get

$$
\left.
   \begin{array}{rcl}
    \mathbb{H}                    & \underrightarrow{\qquad |T|\qquad }     &\mathbb{H}\\
    F_{z^{-1}}^{\mathbb{H}}\downarrow  &                             & \downarrow F_{z^{-1}}^{\mathbb{H}}\\
    \mathbb{H}                    & \overrightarrow{\qquad |T|^{-1} \qquad} &\mathbb{H}
    \end{array}
 \right.
$$

(2) By Definition~\ref{yousuanzidexingshidingyi5}, on $\sigma(|T|^{-1})$ define the function

$$
F_{xx^{*}}:
   \mathcal{L}^{2}(\sigma(|T|^{-1}),\mu_{|T|^{-1}})\rightarrow\mathcal{L}^{2}(\sigma(|T^{-1}|),\mu_{|T^{-1}|}),F_{xx^{*}}(f(xx^{*}))=f(x^{*}x).
$$

On $\sigma(|T|^{-1})$,
by \cite{Hualookang1949} we get that $F_{xx^{*}}$ is a linear algebraic isomorphic
from $\mathcal{P}(xx^{*})$ to $\mathcal{P}(x^{*}x)$.

For any $z\in \sigma(|T^{-1}|)$,
by Lemma $\ref{weierstrasspingfangkefenyinli2}$ we get that $\mathcal{P}(|z|^2)$ is dense in $\mathcal{C}(|z|)$ and
$\mathcal{C}(|z|)$ is dense in $\mathcal{L}^2(\sigma(|T^{-1}|),\mu_{|T^{-1}|})$.
Hence $\mathcal{P}(|z|^2)$ is dense in $\mathcal{L}^2(\sigma(|T^{-1}|),\mu_{|T^{-1}|})$.

With a similarly discussion,
for any $y\in \sigma(|T|^{-1})$, we get that $\mathcal{P}(|y|^2)$ is dense in $\mathcal{L}^2(\sigma(|T|^{-1}),\mu_{|T|^{-1}})$.

With the isomorphic representations $R_{|T|^{-1}}$ and $R_{|T^{-1}|}$ of Theorem$~\ref{hanshuyansuan5}$,
$F_{xx^{*}}$ is an invertible linear operator from $\mathcal{L}^2(\sigma(|T|^{-1}),\mu_{|T|^{-1}})$ to $\mathcal{L}^2(\sigma(|T^{-1}|),\mu_{|T^{-1}|})$.
So
$F_{xx^{*}}\circ F_{z^{-1}}$ is an invertible linear operator from
$\mathcal{L}^2(\sigma(|T|),\mu_{|T|})$
to
$\mathcal{L}^2(\sigma(|T^{-1}|),\mu_{|T^{-1}|})$.

From Lemma~\ref{puxiangdengdedingyi5},
for any $p_n\in\mathcal{P}(\sigma(|T|^{-1}))\subseteq\mathcal{A}(\sigma(|T|^{-1}))$,
because of
$$
T^{*-1}p_n(|T|^{-1})=p_n(|T^{-1}|)T^{*-1}.
$$
By \cite{JohnBConway2000}P60 we get that there is an unitary operator $U\in\mathcal{B}(\mathbb{H})$
such that
$$U\mathcal{P}(|T|^{-1})=\mathcal{P}(|T^{-1}|)U.$$

Hence, $$U\mathcal{A}(|T|^{-1})=\mathcal{A}(|T^{-1}|)U,$$
and
$$U\overline{\mathcal{A}(|T|^{-1})}\xi_{|T|}=\overline{\mathcal{A}(|T^{-1}|)}U\xi_{|T|}.$$

Let $\xi_{|T^{-1}|}=U\xi_{|T|}$, so $\xi_{|T^{-1}|}$ is a $\mathcal{A}(|T^{-1}|)$-cyclic vector and  $\overline{\mathcal{A}(|T^{-1}|)}\xi_{|T^{-1}|}=\mathbb{H}$.
With Theorem$~\ref{hanshuyansuan5}$, we get

\begin{eqnarray*}
&&\int\limits_{\sigma(|T|^{-1})}{f(z)d\mu_{|T|^{-1}}(z)}=<f(|T|^{-1})\xi_{|T|},\xi_{|T|}>
=\int\limits_{\sigma(|T|)}{f(\frac{1}{z})d\mu_{|T|}(z)}\\
\end{eqnarray*}

and
\begin{eqnarray*}
&&\int\limits_{\sigma(|T^{-1}|)}{f(z)d\mu_{|T^{-1}|}(z)}=<f(|T^{-1}|)\xi_{|T^{-1}|},\xi_{|T^{-1}|}>,
\end{eqnarray*}

with $F_{xx^{*}}$ is a linear algebraic isomorphic, we get $[d\mu_{|T^{-1}|}]=[d\mu_{|T|^{-1}}]$, that is,
$d\mu_{|T^{-1}|}$ and $d\mu_{|T|^{-1}}$ are mutually absolutely continuous,
by \cite{JohnBConway1990}IX 3.6 theorem and (1) we get

$$
d\mu_{|T^{-1}|}=|\phi_{|T|}(\frac{1}{z})|d\mu_{|T|^{-1}}=|z|^2|\phi_{|T|}(z)|d\mu_{|T|},
$$

where $|\phi_{|T|}(z)|\neq0,a.e.$ and $\phi_{|T|}(z)\in\mathcal{L}^{1}(\sigma(|T|),\mu_{|T|})$.
So we get

\begin{eqnarray*}
&&\|F_{xx^{*}}\circ F_{z^{-1}}(f(z))\|^2_{\mathcal{L}^2(\sigma(|T^{-1}|),\mu_{|T^{-1}|})}\\
&&=\int\limits_{\sigma(|T^{-1}|)}{F_{xx^{*}}\circ F_{z^{-1}}(f(z))\overline{F_{xx^{*}}\circ F_{z^{-1}}(f(z))}d\mu_{|T^{-1}|}(z)}\\
&&=\int\limits_{\sigma(|T^{-1}|)}{F_{xx^{*}}(f(z^{-1}))\overline{F_{xx^{*}}(f(z^{-1}))}d\mu_{|T^{-1}|}(z)}\\
&&\triangleq\int\limits_{\sigma(|T|^{-1})}{f(y^{-1})\bar{f}(y^{-1})d\mu_{|T|^{-1}}(y)}\\
&&=\int\limits_{\sigma(|T|^{-1})}{F_{y^{-1}}(f(y))F_{y^{-1}}(\bar{f}(y))d\mu_{|T|^{-1}}(y)}\\
&&=\|F_{y^{-1}}(f(y))\|^2_{\mathcal{L}^2(\sigma(|T|^{-1}),\mu_{|T|^{-1}})},
\end{eqnarray*}

where $\triangleq$ is introduced by $U$.

Hence $F_{xx^{*}}$ is an unitary operator from $\mathcal{L}^2(\sigma(|T|^{-1}),\mu_{|T|^{-1}})$
to $\mathcal{L}^2(\sigma(|T^{-1}|),\mu_{|T^{-1}|})$ induced by $U$,
i.e. by$|\phi_{|T|}(\frac{1}{z})|$.

With Theorem~\ref{hanshuyansuan5} and Definition~\ref{yousuanzidexingshidingyi5},
we get
$$
F_{xx^{*}}^{\mathbb{H}}:
   \mathcal{A}(|T|^{-2})\xi_{|T|}\rightarrow\mathcal{A}(|T^{-2}|)\xi_{|T^{-1}|},
   F_{xx^{*}}^{\mathbb{H}}(f(|T|^{-2})\xi_{|T|})=f(|T^{-2}|)\xi_{|T^{-1}|}.
$$

Therefore $F_{xx^{*}}^{\mathbb{H}}$ is an unitary operator from $\overline{\mathcal{A}(|T|^{-1})\xi_{|T|}}$ to $\overline{\mathcal{A}(|T^{-1}|)\xi_{|T^{-1}|}}$.
By Lemma $\ref{gnsfenjiedingliyingyongyinli3}$ we get

$$
\overline{\mathcal{A}(|T|^{-1})\xi_{|T|}}=\mathbb{H}=\overline{\mathcal{A}(|T^{-1}|)\xi_{|T^{-1}|}}.
$$

That is, $F_{xx^{*}}^{\mathbb{H}}$ is an unitary operator and we get

$$
\left.
   \begin{array}{rcl}
    \mathbb{H}                    & \underrightarrow{\qquad |T|^{-1}\qquad }     &\mathbb{H}\\
    F_{xx^{*}}^{\mathbb{H}}\downarrow  &                             & \downarrow F_{xx^{*}}^{\mathbb{H}}\\
    \mathbb{H}                    & \overrightarrow{\qquad |T^{-1}| \qquad} &\mathbb{H}
    \end{array}
 \right.
$$

Hence $F_{xx^{*}}^{\mathbb{H}}|T|^{-1}=|T^{-1}|F_{xx^{*}}^{\mathbb{H}},$
the unitary operator $F_{xx^{*}}^{\mathbb{H}}$ is induced by the function $|\phi_{|T|}(\frac{1}{z})|$,
where $|\phi_{|T|}(z)|\in\mathcal{L}^{1}(\sigma(|T|),\mu_{|T|})$.
\end{proof}

\begin{corollary}\label{TjueduizhiyuTxingjueduizhideguanxi16}
Let $T$ be an invertible bounded linear operator on $\mathbb{H}$ and
let $\mathcal{A}(|T|)$ be the complex $C^{*}$-algebra generated by $|T|$ and $1$.
There is $\sigma(|T|)=\sigma(|T^{*}|)$ and we get that
$$F_{xx^{*}}^{\mathbb{H}}|T|=|T^{*}|F_{xx^{*}}^{\mathbb{H}},$$
where $F_{xx^{*}}^{\mathbb{H}}$ is the unitary operator of Theorem$~\ref{TjueduizhiniyuTnijueduizhideguanxi15}$.
Moreover $F_{xx^{*}}^{\mathbb{H}}$ is induced by an almost everywhere non-zero function $|\phi_{|T|}(z)|$, where $|\phi_{|T|}(z)|\in\mathcal{L}^{1}(\sigma(|T|),\mu_{|T|})$.
That is, $d\,\mu_{|T^{*}|}=|\phi_{|T|}(z)|d\,\mu_{|T|}$.
\end{corollary}

Now, for any given $g(z)\in\mathcal{L}^{\infty}(\sigma(|T|),\mu_{|T|})$,
we define
$$M_g:\mathcal{L}^{2}(\sigma(|T|),\mu_{|T|})\to\mathcal{L}^{2}(\sigma(|T|),\mu_{|T|}),M_gf(z)=g(z)f(z).$$

\begin{theorem}\label{Tjifenjiebiaoshi17}
Let $T$ be an invertible bounded linear operator on $\mathbb{H}$,
$T=U|T|$ be the Polar Decomposition, and $M_z$ be the unitary equivalent representation of $|T|$ on $\mathcal{L}^{2}(\sigma(|T|),\mu_{|T|})$,
i.e. on $\sigma(|T|)$ give the functional calculate of $|T|$,
where $M_zf(z)=zf(z)$ for any $f(z)\in\mathcal{L}^{2}(\sigma(|T|),\mu_{|T|})$.
There is $\psi(z)\in\mathcal{L}^{\infty}(\sigma(|T|),\mu_{|T|})$ such that
$F_{xx^{*}}M_{z\psi(z)}$ and $F_{xx^{*}}M_{\psi(z)}$ are the unitary equivalent representations of $T$ and $U$
on $\mathcal{L}^{2}(\sigma(|T|),\mu_{|T|})$,
respectively,
where $M_{z\psi(z)}=M_zM_{\psi(z)}=M_{\psi(z)}M_z$.
\end{theorem}

\begin{proof}
By the Polar Decomposition $T=U|T|$,
we get that $T^{*}T=|T|^{2}$ and $TT^{*}=U|T|^{2}U^{*}$.
By Corollary $\ref{TjueduizhiyuTxingjueduizhideguanxi16}$
we get $TT^{*}=F_{xx^{*}}^{\mathbb{H}}|T|^{2}(F_{xx^{*}}^{\mathbb{H}})^{*}$.
With Theorem$~\ref{TjueduizhiniyuTnijueduizhideguanxi15}$
and the fact that $\{M_{\psi(z)}:\psi(z)\in\mathcal{L}^{\infty}(\sigma(|T|),\mu_{|T|})\}$ is a maximal abelian von Neumann algebra in $\mathcal{B}(\mathcal{L}^{2}(\sigma(|T|),\mu_{|T|}))$,
we get the proof from the Corollary of the Fuglede-Putnam theorem \cite{JohnBConway1990}P279.
\end{proof}

\begin{corollary}\label{yibanTjifenjiebiaoshi18}
Let $T$ be a bounded linear operator on $\mathbb{H}$ and
$T=U|T|$ be the Polar Decomposition.
There is $\psi(z)\in\mathcal{L}^{\infty}(\sigma(|T+a|),\mu_{|T+a|})$ such that
$$F_{xx^*}(M_{z\psi(z)}-a)$$
is the unitary equivalent representation of $T$ on $\mathcal{L}^{2}(\sigma(|T+a|),\mu_{|T+a|})$.
\end{corollary}

\begin{proof}
For $a\in\rho(T)$, $S=T+a$ is an invertible bounded linear operator on $\mathbb{H}$,
where $\rho(T)=\mathbb{C}\setminus\sigma(T)$.
With Theorem$~\ref{TjueduizhiniyuTnijueduizhideguanxi15}$,
$F_{xx^{*}}^{\mathbb{H}}$ is an unitary operator and we get

$$
SS^{*}=F_{xx^{*}}^{\mathbb{H}}|S|^{2}(F_{xx^{*}}^{\mathbb{H}})^{*}.
$$

By the proof of Theorem $\ref{Tjifenjiebiaoshi17}$ we get that $F_{xx^*}M_{z\psi(z)}$ is the unitary equivalent representation of $T+a$ on $\mathcal{L}^{2}(\sigma(|T+a|),\mu_{|T+a|})$.
For any $a\in\mathbb{C}$ we get
$$F_{xx^{*}}^{\mathbb{H}}(a)=a.$$

Therefore,
$F_{xx^*}(M_{z\psi(z)}-a)$ is the unitary equivalent representation of $T$ on $\mathcal{L}^{2}(\sigma(|T+a|),\mu_{|T+a|})$
\end{proof}

So if some properties of operators on $\mathbb{H}$ only depending the norm which is compatible with the inner product, then this properties only depending the corresponding properties of elements in
$$
\{M_{\psi_n(z)}:\psi_n(z)\in\mathcal{L}^{\infty}(\sigma(|T^n+a_n|),\mu_{|T^n+a_n|}),a_n\in\rho(T^n),n\in\mathbb{N}\},
$$
just keeping the unitary equivalent representation in mind.
Therefor the unitary equivalent representation is a powerful facility to deal with this problems.
Then we give some properties of normal operator through the unitary equivalent representation.

\begin{corollary}\label{Tzhengguisuanzijifenjiebiaoshi19}
Let $T$ be a bounded linear operator on $\mathbb{H}$ and
$TT^{*}=T^{*}T$.
There is $\psi(z)\in\mathcal{L}^{\infty}(\sigma(|T+a|),\mu_{|T+a|})$ such that $M_{\psi(z)}$ is the unitary equivalent representation of $T$ on $\mathcal{L}^{2}(\sigma(|T+a|),\mu_{|T+a|})$.
That is $T\in\mathcal{A}^{'}(|T|)$,
where $a\in\rho(T)$ and
$$\mathcal{A}^{'}(|T|)=\{A\in\mathcal{B}(\mathbb{H}): AB=BA \text{ for every } B\in\mathcal{A}(|T|)\}.$$
\end{corollary}

\begin{corollary}\label{Tzhengguisuanzijifenjiebiaoshi20}
Let $T$ be a bounded linear operator on $\mathbb{H}$,
such that $T$ is normal if and only if $T$ is unitary equivalent to $M_{\psi(z)}$ on $\mathcal{L}^{2}(\sigma(|T+a|),\mu_{|T+a|})$, and if and only if $T\in\mathcal{A}^{'}(|T|)$,
where $a\in\rho(|T|)$,$\psi(z)\in\mathcal{L}^{\infty}(\sigma(|T+a|),\mu_{|T+a|})$.
\end{corollary}

\section{Invariant subspace problem}

In this part, we study the invariant subspace problem on Hilbert spaces.
With the fact that the exist of nontrivial invariant subspace is unchanged by similarly between bounded linear operators on Banach spaces \cite{JohannesKaiser2016}.

That is, for $R,S,T\in\mathcal{B}(\mathbb{H})$ and $T$ is invertible,
if $R=TST^{-1}$,
then $R$ has nontrivial invariant subspace if and only if $S$ has.

Therefore for any given $T\in\mathcal{B}(\mathbb{H})$,
we can using the unitary equivalent representation of $T$ to solve the invariant subspace problem of $T$.
Using the construction of $F_{xx^*}$,
we give a sufficient condition to estimate this problem.

For convenience define $Fix(F_{xx^*}^{\mathbb{H}})=\{F_{xx^*}^{\mathbb{H}}(f)=f;f\in\mathbb{H}\}$,
similarly we can define $Fix(F_{xx^*})$.
Obviously, $Fix(F_{xx^*}^{\mathbb{H}})$ is a subspace of $\mathbb{H}$.
\begin{theorem}\label{TyujueduizhiTdebubianzikongjian21}
Let $\dim{\mathbb{H}}>1$,
$T\in\mathcal{B}(\mathbb{H})$ and
$F_{xx^*}(M_{z\psi(z)}-a)$ is the unitary equivalent representation of $T$ on $\mathcal{L}^{2}(\sigma(|T+a|),\mu_{|T+a|})$.
If ${z\psi(z)}\in Fix(F_{xx^*})$,
then $T$ has a nontrivial invariant subspace space .
\end{theorem}

\begin{proof}
Only need to prove this for infinite dimension separable complex Hilbert space $\mathbb{H}$.
Obviously, if $A\subset\mathbb{H}$ is a nontrivial invariant subspace of $T$
if and only if $A$ is a nontrivial invariant subspace of $T+a$, where $a\in\mathbb{C}$.

Let $a\in\rho(T)$,
then with Corollary $\ref{yibanTjifenjiebiaoshi18}$
there is $\psi(z)\in\mathcal{L}^{\infty}(\sigma(|T+a|),\mu_{|T+a|})$ such that $F_{xx^*}(M_{z\psi(z)}-a)$ is the unitary equivalent representation of $T$ on $\mathcal{L}^{2}(\sigma(|T+a|),\mu_{|T+a|})$.

With the construction $F_{xx^{*}}$ of Theorem $\ref{TjueduizhiniyuTnijueduizhideguanxi15}$,
we get $(F_{xx^{*}})^2=identity$ and $Fix(F_{xx^*})\neq\emptyset$.

(1) If $Fix(F_{xx^*}^{\mathbb{H}})=\mathbb{H}$,
that is $F_{xx^{*}}=identity$,
then $M_{z\psi(z)}-a$ is a normal operator.
By\cite{JohnBConway1990},
The invariant subspace problem is trivial for normal operator.

(2) If $Fix(F_{xx^*}^{\mathbb{H}})\neq\mathbb{H}$,
then $Fix(F_{xx^*})$ is a nontrivial invariant subspace of $F_{xx^*}$,
for ${z\psi(z)}\in Fix(F_{xx^*})$,
we get
$$
F_{xx^*}(M_{z\psi(z)}-a)Fix(F_{xx^*})\subseteq Fix(F_{xx^*}),
$$

that is, $Fix(F_{xx^*})$ is a nontrivial invariant subspace of $F_{xx^*}(M_{z\psi(z)}-a)$.
\end{proof}

\section{Lebesgue operator}

In this part, using our construction $F_{xx*}(M_{z\psi(z)}-a)$ of the unitary equivalent representations of $T$ on $\mathcal{L}^{2}(\sigma(|T+a|),\mu_{|T+a|})$,
we study chaos of invertible bounded linear operator on Hilbert spaces.

For the example of singular integral in mathematical analysis,
we know that is independent the convergence or the divergence of the weighted integral between  $x$ and $x^{-1}$,
however some times that indeed dependent for a special weighted function.

In the view of the singular integral and by our construction of $F_{xx*}(M_{z\psi(z)}-a)$,
we define Lebesgue class
and prove that $T$ and $T^{*-1}$ are Li-Yorke chaotic at the same time when $T$ is a Lebesgue operator.
Also, we give an example such that $T$ is a Lebesgue operator, but not is a normal operator.

Let $dx$ be the Lebesgue measure on $\mathcal{L}^{2}(\mathbb{R}_{+})$.
With Theorem~\ref{hanshuyansuan5} we get $d\mu_{|T^n|}$ is the complete Borel measure
and $\mathcal{L}^2(\sigma(|T^n|),d\mu_{|T^n|})$ is a Hilbert space.

If exists $N>0$,
for any $n\geq N$, $d\mu_{|T^n|}$ is absolutely continuity with respect to $dx$,
where $n\in\mathbb{N}$.
By the Radon-Nikodym Theorem \cite{JohnBConway1990}P380 there is $f_n\in\mathcal{L}^{1}(\mathbb{R}_{+})$ such that $d\mu_{|T^n|}=f_n(x)\,dx$.

\begin{definition}\label{lebesguesuanzileidingyi6}
Let $T$ be an invertible bounded linear operator on the separable Hilbert space $\mathbb{H}$ over $\mathbb{C}$,
if $T$ satisfies the following assertions:

$(1)$ If $\exists N>0$, for $\forall n\geq N$, $n\in\mathbb{N}$
\begin{eqnarray*}
\left\{\begin{array}{lr}
d\mu_{|T^n|}=f_n(x)\,dx,& f_n\in\mathcal{L}^{1}(\mathbb{R}_{+}).\\
\\
x^{2}f_n(x)=f_{n}(x^{-1}),& 0<x\leq1.
\end{array}\right.
\end{eqnarray*}

$(2)$ If $\exists N>0$,
for $\forall n\geq N,n\in\mathbb{N}$, there is a $\mathcal{A}(|T^n|)$-cyclic vector $\xi_n$.
And for any given $0\neq x\in\mathbb{H}$ and for any given $0\neq g_n(t)\in\mathcal{L}^2(\sigma(|T^n|),d\mu_{|T^n|})$,
there is an unique $0\neq y\in\mathbb{H}$ such that $y=g_n(|T^n|^{-1})\xi_n$ whenever $x=g_n(|T^n|)\xi_n$.

Then we say that $T$ is a Lebesgue operator,
let $\mathcal{B}_{Leb}(\mathbb{H})$ denote the set of all Lebesgue operators.
\end{definition}

\begin{theorem}\label{lebesguesuanzihundundeduichengxing7}
Let $T$ be a Lebesgue operator on the separable Hilbert space $\mathbb{H}$ over $\mathbb{C}$,
then $T$ is Li-Yorke chaotic if and only if $T^{*-1}$ is.
\end{theorem}
\begin{proof}
Let $\mathbb{H}$ be $\mathcal{A}(|T^n|)$-cyclic,
that is,
there is a vector $\xi_n$ such that $\overline{\mathcal{A}(|T^n|)\xi_n}=\mathbb{H}$.
If $x_0$ be a Li-Yorke chaotic point of $T$,
by the define of Lebesgue operator,
we get that for enough large $n\in\mathbb{N}$,
there are $g_n(x)\in\mathcal{L}^2(\sigma(|T^n|),d\mu_{|T^n|})$,
$f_n(x)\in\mathcal{L}^{2}(\mathbb{R}_{+})$ and $y_0\in\mathbb{H}$
such that $x_0=g_n(|T^n|)\xi_n$,
$y_0=g_n(|T^n|^{-1})\xi_n$ and $d\mu_{|T^n|}=f_n(x)\,dx$.
Therefore

\begin{eqnarray*}
&&\|T^nx_0\|\\
&&=<T^{n*}T^nx_0,x_0>\\
&&=<|T^n|^2g_n(|T^n|)\xi_n,g_n(|T^n|)\xi_n>\\
&&=<g_n(|T^n|)^{*}|T^n|^2g_n(|T^n|)\xi_n,\xi_n>\\
&&=\displaystyle\int\limits_{\sigma(|T^n|)}{x^2g_n(x)\bar{g}(x)\,d\mu_{|T^n|}(x)}\\
&&=\displaystyle\int_{0}^{+\infty}{x^2|g_n(x)|^2f_n(x)\,dx}\\
&&=\displaystyle\int_{0}^{1}{x^2|g_n(x)|^2f_n(x)\,dx}+\displaystyle\int_{1}^{+\infty}{x^2|g_n(x)|^2f_n(x)\,dx}\\
&&=\displaystyle\int_{0}^{1}{x^2|g_n(x)|^2f_n(x)\,dx}+\displaystyle\int_{0}^{1}{x^{-4}|g_n(x^{-1})|^2f_n(x^{-1})\,dx}\\
&&\triangleq
\displaystyle\int_{0}^{1}{|g_n(x)|^2f_n(x^{-1})\,dx}+\displaystyle\int_{0}^{1}{x^{-2}|g_n(x^{-1})|^2f_n(x)\,dx}\\
&&=\displaystyle\int_{1}^{+\infty}{x^{-2}|g_n(x^{-1})|^2f_n(x)\,dx}+\displaystyle\int_{0}^{1}{x^{-2}|g_n(x^{-1})|^2f_n(x)\,dx}\\
&&=\displaystyle\int_{0}^{+\infty}{x^{-2}|g_n(x^{-1})|^2f_n(x)\,dx}\\
&&=\displaystyle\int\limits_{\sigma(|T^n|)}{x^{-2}g_n(x^{-1})\bar{g}_n(x^{-1})\,d\mu_{|T^n|}(x)}\\
&&=<g_n(|T^n|)^{*}|T^n|^{-2}g_n(|T^n|^{-1})\xi_n,\xi_n>\\
&&=<|T^n|^{-2}g_n(|T^n|^{-1})\xi_n,g_n(|T^n|^{-1})\xi_n>\\
&&=<|T^n|^{-2}y_0,y_0>\\
&&=<T^{-n}T^{-n*}y_0,y_0>\\
&&=\|T^{*-n}y_0\|.
\end{eqnarray*}

Where $\triangleq$ following the define of $f_n(x)$.
By the Li-Yorke chaos definition, we get that $T$ is Li-Yorke chaotic if and only if $T^{*-1}$ is Li-Yorke chaotic.
\end{proof}
Define a distributional function $F_{x}^{n}(\tau)=\frac{1}{n}\sharp\{0\leq i\leq n:\|T^n(x)\|<\tau\}$,
where $T\in\mathcal{B}(\mathbb{H}),x\in\mathbb{H},n\in\mathbb{N}$.
And define

$$
F_{x}(\tau)=\liminf\limits_{n\to\infty} F_{x}^{n}(\tau),\qquad
F_{x}^{*}(\tau)=\limsup\limits_{n\to\infty} F_{x}^{n}(\tau).
$$

\begin{definition}\label{fenbuhundundedingyi2}
Let $T\in\mathcal{B}(\mathbb{H})$, if there exists $x\in\mathbb{H}$ and

(1) If $F_{x}(\tau)=0,\exists\tau>0$,
and $F_{x}^{*}(\epsilon)=1,\forall\epsilon>0$,
then we say that $T$ is distributional chaotic or $I$-distributionally chaotic.

(2) If $F_{x}^{*}(\epsilon)>F_{x}(\tau),\forall\tau>0$,
and $F_{x}^{*}(\epsilon)=1,\forall\epsilon>0$,
then we say that $T$ is $II$-distributionally chaotic.

(3) If $F_{x}^{*}(\epsilon)>F_{x}(\tau),\forall\tau>0$, then we say that $T$ is $III$-distributionally chaotic.
\end{definition}

\begin{corollary}\label{lebesguesuanzifenbuhundundengjia8}
Let $T$ be a Lebesgue operator on the separable Hilbert space $\mathbb{H}$ over $\mathbb{C}$,
then $T$ is $I$-distributionally chaotic (or $II$-distributionally chaotic or $III$-distributionally chaotic) if and only if $T^{*-1}$ is $I$-distributionally chaotic (or $II$-distributionally chaotic or $III$-distributionally chaotic).
\end{corollary}

\begin{theorem}\label{lebesguesuanzicunzaixing9}
There is an invertible bounded linear operator $T$ on the separable Hilbert space $\mathbb{H}$ over $\mathbb{C}$,
$T$ is Lebesgue operator but not is a normal operator.
\end{theorem}

\begin{proof}
Let $0<a<b<+\infty$,
then $\mathcal{L}^2([a,b])$ is a separable Hilbert space over $\mathbb{R}$,
because any separable Hilbert space over $\mathbb{R}$ can be expanded to a separable Hilbert space over $\mathbb{C}$,
it is enough to prove the conclusion on $\mathcal{L}^2([a,b])$.
We prove the conclusion by six parts:

(1) Let $0<a<1<b=\frac{1}{a}<+\infty$,
construct measure preserving transformation on $[a,b]$.

Let $M=\{[a,\frac{b-a}{2}],[\frac{b-a}{2},b]\}$.
There is a Borel algebra $\xi(M)$ generated by $M$,
define $\Phi:[a,b]\to[a,b]$,

$$
\Phi([a,\frac{b-a}{2}])=[\frac{b-a}{2},b],\qquad
\Phi([\frac{b-a}{2},b])=[a,\frac{b-a}{2}].
$$

Then $\Phi$ is an invertible measure preserving transformation on the Borel algebra $\xi(M)$.
By \cite{PeterWalters1982}P63,
$U_{\Phi}\neq1$ and $U_{\Phi}$ is a unitary operator induced by $\Phi$,
where $U_{\Phi}$ is the composition $U_{\Phi}h=h\circ\Phi$,
for any $h\in\mathcal{L}^2([a,b])$.

(2) Define $M_xh=xh$ on $\mathcal{L}^2([a,b])$,
then $M_x$ is an invertible positive operator.

(3) For $f(x)=\cfrac{|\ln\,x|}{x}$, $x>0$,
define $d\mu=f(x)d\,x$.
then $f(x)$ is continuous and $f(x)>0$, a.e. $x\in[a,b]$,
hence $d\,\mu$ that is absolutely continuous with respect to $d\,x$ is finite positive complete Borel measure
and $\mathcal{L}^2([a,b],d\,\mu)$ is a separable Hilbert space over $\mathbb{R}$.
Moreover $\mathcal{L}^2([a,b])$ and $\mathcal{L}^2([a,b],d\,\mu)$ are unitary equivalent.

(4) Let $T=U_{\Phi}M_x$,
we get

$$
T^{*}T=U_{\Phi}TT^{*}U_{\Phi}^{*},\qquad U_{\Phi}\neq1.
$$

Because of
$$
U_{\Phi}M_{x}\neq M_{x}U_{\Phi},\qquad U_{\Phi}M_{x^2}\neq U_{\Phi}M_{x^2},
$$
we get that $T$ is not a normal operator and $\sigma(|T|)=[a,b]$.

(5) The operator $T=U_{\Phi}M_x$ on $\mathcal{L}^2([a,b])$ is corresponding to the operator $T^{\prime}$ on $\mathcal{L}^2([a,b],d\,\mu)$, $T^{\prime}$ is invertible bounded linear operator and is not a normal operator,
$\sigma(|T^{\prime}|)=[a,b]$.

(6) From

$$
\displaystyle\int_{a}^{b}x^{n}f(x)d\,x=\int_{a^n}^{b^n}tf(t^{\tfrac{1}{n}})\frac{1}{nt^{\tfrac{n-1}{n}}}d\,t,
$$

let

$$
f_n(t)=\frac{1}{n}I_{[a^n,b^n]}f(t^{\tfrac{1}{n}})\frac{1}{t^{\tfrac{n-1}{n}}},
$$

we get that $f_n(t)$ is continuous and almost everywhere positive,
hence $f_n(t)d\,t$ is a finite positive complete Borel measure.

For any $E\subseteq\mathbb{R}_{+}$ define $I_{E}=1$ when $x\in E$
else $I_{E}=0$,
so $I_{E}$ is the identity function on $E$.

\scalebox{1.2}[1.2]
{$
\left.\begin{array}{l}
(i) f_n(t^{-1})=\frac{1}{n}I_{[a^n,b^n]}f(t^{-\tfrac{1}{n}})\frac{1}{t^{-\tfrac{n-1}{n}}}
=\frac{1}{n}I_{[a^n,b^n]}\frac{|\ln\,t^{-\tfrac{1}{n}}|}{t^{-\tfrac{1}{n}}}\frac{1}{t^{-\tfrac{n-1}{n}}}\\
=\frac{1}{n}I_{[a^n,b^n]}t|\ln\,t^{\tfrac{1}{n}}|.
\end{array}\right.
$}

\scalebox{1.2}[1.2]
{$
\left.\begin{array}{l}
(ii) t^2f_n(t)=\frac{1}{n}I_{[a^n,b^n]}f(t^{\tfrac{1}{n}})\frac{t^2}{t^{\tfrac{n-1}{n}}}
=\frac{1}{n}I_{[a^n,b^n]}\frac{|\ln\,t^{\tfrac{1}{n}}|}{t^{\tfrac{1}{n}}}\frac{t^2}{t^{\tfrac{n-1}{n}}}\\
=\frac{1}{n}I_{[a^n,b^n]}t{|\ln\,t^{\tfrac{1}{n}}|}.
\end{array}\right.
$}

By $(i)(ii)$ we get $x^2f_{n}(x)=f_{n}(x^{-1})$.
From $\sigma(|T^{\prime n}|)=[a^n,b^n]$ and
$$
\int_{a^n}^{b^n}t^2f(t^{\tfrac{1}{n}})\frac{1}{nt^{\tfrac{n-1}{n}}}d\,t=\int_{0}^{+\infty}t^2f_{n}(t)dt,
$$
let $d\,\mu_{|T^{\prime n}|}=f_n(t)d\,t$,
then $d\,\mu_{|T^{\prime n}|}$ is the finite positive complete Borel measure.

For any given $0\neq h(x)\in\mathcal{L}^2([a,b])$ we get $0\neq h(x^{-1})\in\mathcal{L}^2([a,b])$.
$I_{[a,b]}$ is a $\mathcal{A}(|M_{x}^{n}|)$-cyclic vector of the multiplication $M_{x}^{n}=M_{x^n}$,
and $I_{[a^n,b^n]}$ is a $\mathcal{A}(|T^{\prime n}|)$-cyclic vector of $|T^{\prime n}|$,
By Definition $\ref{lebesguesuanzileidingyi6}$ we get that $T^{\prime}$ is Lebesgus operator but not is a normal operator.
\end{proof}

\begin{corollary}\label{lebesguesuanzicunzaizhengguisuanzi10}
There is an invertible bounded linear operator $T$ on the separable Hilbert space $\mathbb{H}$ over $\mathbb{C}$,
$T$ is a Lebesgue operator and also is a positive operator.
\end{corollary}

\begin{corollary}\label{lebesguesuanziyuzhengguisuanzidejiaofeikong11}
There is
$$\mathcal{B}_{Leb}(\mathbb{H})\cap\mathcal{B}_{Nor}(\mathbb{H})\neq\emptyset,$$
and
$$\mathcal{B}_{Leb}(\mathbb{H})\cap(\mathcal{B}(\mathbb{H})\setminus\mathcal{B}_{Nor}(\mathbb{H}))\neq\emptyset,$$

where $\mathcal{B}_{Nor}(\mathbb{H})$ is the set of all normal operators on $\mathbb{H}$.
\end{corollary}

By the construction of the unitary operator $F_{xx*}$ in this paper,
we could study the operator by the integral on $\mathbb{R}$.
This way maybe neither change Li-Yorke chaotic nor the computing,
but by the integral in mathematical analysis on the theoretical level we should find
some operator class and study its properties, as we introduce the Lebesgue operator in this part.

\bibliographystyle{amsplain}

\end{document}